\documentclass[a4paper,10pt]{amsart}
\usepackage{verbatim}
\usepackage{amsmath}
\usepackage{amssymb,color}

\usepackage[colorlinks]{hyperref}
\definecolor{linkblue}{RGB}{1,1,190}
\definecolor{citered}{RGB}{190,1,1}
\hypersetup{linkcolor=linkblue,urlcolor=linkblue,citecolor=citered}

\usepackage{mathtools}
\usepackage{tabularx}
\usepackage{multirow}
\usepackage{xcolor}

\theoremstyle{plain}
\newtheorem{theorem}{\bf Theorem}[section]
\newtheorem{proposition}[theorem]{\bf Proposition}
\newtheorem{lemma}[theorem]{\bf Lemma}
\newtheorem{corollary}[theorem]{\bf Corollary}

\theoremstyle{definition}

\newtheorem{remark}[theorem]{\bf Remark}

\newcommand{\N}{\mathbb N}

\DeclareMathOperator{\ord}{ord}

\DeclareMathOperator{\supp}{supp}

\numberwithin{equation}{section}

\thanks{This work was supported by the Austrian Science Fund FWF (Project Number W1230) and by the ANR (project ANR-21-
CE39-0009 - BARRACUDA)}

\begin{document}\title[Sets of Cross Numbers]{Sets of Cross Numbers of Sequences over Finite Abelian Groups}

\author{Aqsa Bashir \and Wolfgang A. Schmid}

\address{University of Graz, NAWI Graz\\
Department of Mathematics and Scientific Computing\\
Heinrichstra{\ss}e 36\\
8010 Graz, Austria}
\email{aqsa.bashir@uni-graz.at}

\address{Laboratoire Analyse, G{\'e}om{\'e}trie et Applications, LAGA, Universit{\'e} Sorbonne Paris Nord, CNRS, UMR 7539, F-93430, Villetaneuse, France
 \\ and \\ Laboratoire Analyse, G{\'e}om{\'e}trie et Applications (LAGA, UMR 7539) \\ COMUE  Universit{\'e} Paris Lumi{\`e}res \\  Universit{\'e} Paris 8, CNRS \\  93526 Saint-Denis cedex, France} \email{schmid@math.univ-paris13.fr, wolfgang.schmid@univ-paris8.fr}

\begin{abstract}
Let $G$ be a finite abelian group with $\exp(G)$ the exponent of $G$. Then $\mathsf W(G)$ denotes the set of cross numbers of minimal zero-sum sequences over $G$ and $\mathsf w(G)$ denotes the set of all cross numbers of non-trivial zero-sum free sequences over $G$. 
It is clear that  $\mathsf W(G)$ and $\mathsf w(G)$ are bounded subsets of  $\frac{1}{\exp(G)}\mathbb{N}$ with maximum $ \mathsf K(G)$ and $\mathsf k(G)$, respectively (here $\mathsf{K}(G)$ and $\mathsf{k}(G)$  denote the large and the small cross number of $G$, respectively). 
We give results on the structure of $\mathsf W(G)$ and $\mathsf w(G)$. 
We first show that both sets contain long arithmetic progressions and that only close to the maximum there might be some gaps. 
Then, we provide groups for which $\mathsf W(G)$ and $\mathsf w(G)$ actually are arithmetic progressions, and argue that this is rather a rare phenomenon.  
Finally, we provide some results in case there are gaps.  	
\end{abstract}
\maketitle

\smallskip
\section{Introduction}\label{1}
\smallskip

Let $(G, +)$ be a finite abelian group. For a sequence $S = g_1 \dots g_{\ell}$ over $G$, that is a collection of elements $g_i$ of $G$ where repetitions are allowed (for a formal definition and other undefined terminology see below), the cross number of $S$ is defined as 
\[
\mathsf{k}(S) = \frac{1}{\ord(g_1)} + \dots + \frac{1}{\ord(g_{\ell})}.\]
This can be seen as a weighted version of the length $\ell$ of the sequence. The term cross number was introduced by Krause in 1984 \cite{Kr84, Kr-Za91}.
The cross number is an interesting zero-sum constant and in particular plays an important role in the factorization theory of Krull monoids, see for example \cite{Ge93a, Ge94b, Ge-Sc94, Ge-Sc96, Ge-Sc97,  El-Hu05, Gi08b, Gi09a, Ge-Gr09c, Gi12b, Kr13a, He14a, Ki15a,  Wa20a, Bu-Hu24a} for various contributions on the cross number.   
 
In analogy with the Davenport constant of a finite abelian group, one defines the large cross number of $G$, denoted $\mathsf{K}(G)$ as the maximal cross number of a minimal zero-sum sequence, and the small cross number of $G$, denoted $\mathsf{k}(G)$, as the maximal cross number of a zero-sum free sequence. 
Recall that the large Davenport constant of a finite abelian group  $G$, denoted $\mathsf{D}(G)$, is defined as the maximal length of a minimal zero-sum sequence of $G$ while the small Davenport constant, denoted $\mathsf d(G)$, is the maximal length of a zero-sum free sequence. 

For the small Davenport constant it is obvious that zero-sum free sequences of each length up to the maximum exist as every subsequence of a zero-sum free sequence is zero-sum free.
Also, for the large Davenport constant one can see without difficulty that for each length smaller than the large Davenport constant there exists a minimal zero-sum sequence of that length as well; it suffices to replace two elements in a minimal zero-sum sequence by their sum to obtain a minimal zero-sum sequence whose length is diminished by $1$. 

By contrast, for the cross number the analogous problem is more subtle. Of course, we cannot obtain all the rational numbers up to the maximum as cross numbers, as the cross number of each sequence over $G$ is easily seen to be an integral multiple of $\frac{1}{\exp(G)}$. The actual question is if each integral multiple of $\frac{1}{\exp(G)}$ up to the maximum occurs as a cross number. The basic reasoning that we just recalled for the Davenport constant does not translate to this situation, and indeed in earlier work on the subject it was proved that, depending on the group, there are certain restrictions on the values of the cross numbers of minimal zero-sum sequences, see   \cite{Ch-Ge96, B-C-M-P04}. In recent literature, the sets of values attained by arithmetic invariants (including the sets of elasticities, of catenary degrees, and more) found wide attention, because knowing all the values of the invariants allows a finer understanding of the problem than just knowing its maximum, see for example  \cite{Ge-Zh16a, Ba-Ne-Pe17a, Fa-Ge19a, Ge-Zh19a, Zh19a}.

In the present paper we continue the investigations on the cross numbers of minimal zero-sum sequences and in addition study the values of cross numbers of zero-sum free sequences. The latter problem did not yet get much attention.  It turns out that, depending on the group, the two sets can be essentially identical, with one being a shift of the other by $\frac{1}{\exp(G)}$, yet they can also be quite different.

\section{Preliminaries}
\label{sec_prel}

Let $\mathbb{N}$ denote the set of positive integers  and let $\mathbb{N}_0=\mathbb{N}\cup \{0\}$. For integers $a, b \in \mathbb{Z}$, we set $[a,b]=\{z \in \mathbb{Z} \mid a\leq z\leq b \}$ the interval of integers. Given subsets $A,B \subseteq \mathbb{Z}$ and $\lambda \in \mathbb{R}$, we set $A + B = \{a+b \mid a \in A, \, b\in B \}$ the sumset of $A$ and $B$, and we set  $\lambda A=\{\lambda a \mid a \in A\}$ the dilation of $A$ by $\lambda$; since we hardly use it, we do not introduce a notation for the $s$-fold sumset of $A$, that is the sumset of $s$ copies of $A$.

For $n \in \mathbb{N}$ let $C_n$ be a cyclic group of order $n$; we use additive notation. Let $(G, +, 0)$ be an additive finite abelian group. There exist unique integers $1 < n_1 \mid \dots \mid n_r$ such that $G \cong C_{n_1} \oplus \dots \oplus C_{n_r}$. Moreover, there exist unique prime-powers $1 < q_1 \le  \dots \le  q_s$ such that  $G\cong C_{q_1} \oplus \dots \oplus C_{q_s}$.  The integer $n_r$ is called the exponent of the group, and the order $\ord(g)$ of each element $g \in G$ divides $\exp(G)$. 
One calls $r$ the rank of $G$ and $s$ the total rank of $G$. In case $\exp(G)$ is a prime-power, the two coincide and one calls $G$ a $p$-group. A $p$-group is called an elementary $p$-group if the exponent is prime. 
For a given prime $p$ one calls the number of prime powers $q_i$ that are a $p$-power the $p$-rank of $G$, denoted by $\mathsf{r}_p(G)$.
It is non-zero if and only if the prime $p$ divides the exponent.  

Note that if $|G| = 1$, then the exponent is $1$, while all the ranks  are $0$.

As indicated in the introduction, the focus of the paper is on cross numbers of sequences over a finite abelian group $G$. Informally,  a sequence over a finite abelian group $G$  is a collection of elements of $G$ where repetitions of elements are allowed yet the ordering of the terms is typically disregarded. 
Formally, a sequence over $G$ is an element of the free abelian monoid $\mathcal{F}(G)$ over $G$. 
That is, a sequence $S$ over $G$ can be written uniquely as $S = \prod_{g \in G} g^{v_g}$ with $v_g \in \mathbb{N}_0$ for each $g \in G$; moreover it can be written as $S = g_1 \dots g_{\ell} $ where $g_i \in G$ for each $i \in [1, \ell]$ and these elements are unique up to ordering. The neutral element of $\mathcal{F}(G)$ is denoted by $1$, unless there is a risk of confusion, and is called the trivial sequence. Sometimes we want to consider only sequences that contain elements in a subset $G_0 \subseteq G$; in that case we use the notation $\mathcal{F}(G_0)$. 

One calls 
\begin{itemize}
\item  $\sigma(S) = \sum_{g \in G}v_gg=  \sum_{i=1}^{\ell}g_i$ the sum of $S$,  
\item  $|S| = \sum_{g \in G}v_g=  \ell $ the length of $S$, 
\item $\mathsf{k}(S) = \sum_{g \in G}\frac{v_g}{\ord(g)}=  \sum_{i=1}^{\ell}\frac{1}{\ord(g_i)}$ the cross number of $S$.
\end{itemize}

For a sequence $S$, a subsequence $T$ of $S$ is a divisor of $S$ in $\mathcal{F}(G)$.  We denote by $\Sigma(S) = \{\sigma (T) \colon 1 \neq T \mid S\}$ the set of sums of non-trivial subsequences of $S$. Moreover, $\supp(S) = \{g_1, \dots, g_{\ell}\}$, called the support of $S$,  denotes the set of elements that occur in $S$.  
A sequence $S$ is called a zero-sum sequence if $\sigma (S) = 0$ and it is called zero-sum free if $0 \notin \Sigma (S)$. 
A non-trvial zero-sum sequence  is called a minimal zero-sum sequence if it has no proper non-trivial subsequence that is a zero-sum sequence. 
Equivalently, $S = g_1 \dots g_{\ell}$ is a zero-sum sequence if $\sum_{i= 1}^{\ell}g_i = 0$, while it is zero-sum free if $\sum_{i \in I} g_i \neq 0$ for each $\emptyset \neq I \subseteq [1,\ell]$ (note that $I = [1, \ell]$ is possible). Furthermore, it is a  minimal zero-sum sequence if it is non-trivial with $\sum_{i=1}^{\ell}g_i = 0$, yet $\sum_{i\in I} g_i\neq 0$ for each proper subset $\emptyset \neq I \subsetneq [1,\ell]$.

We denote by $\mathcal{B}(G)$ the set of all zero-sum sequences over $G$, by $\mathcal{A}(G)$ the set of all minimal zero-sum sequences over $G$ and by $\mathcal A^{\ast}(G)$ the set of all non-trivial zero-sum free sequences over $G$ (note that the trivial sequence  is also zero-sum free according to the definition given above).

While it is not central for our current investigation, we mention in passing that  $\mathcal{B}(G)$  is a submonoid of $\mathcal{F}(G)$; this submoind is atomic and its irreducible elements are precisely the elements of $\mathcal{A}(G)$.  

The central objects for the current paper are the following two sets. 
The set of cross numbers of all minimal zero-sum sequences over $G$, denoted
\[
\mathsf W(G)=\{\mathsf k(S) \mid S \in \mathcal{A}(G)\}
\]
has been investigated, e.g., in \cite{B-C-M-P04, Ch-Ge96, Ge-Sc97}. 
We also study the related set  
\[
\mathsf w(G)=\{\mathsf k(S) \mid S \in \mathcal{A}^{\ast}(G)\}
\]
the set of cross numbers of all non-trivial zero-sum free sequences over $G$; 
sometimes it is technically advantageous to consider $\mathsf{w}^{\circ}(G)=  \mathsf{w}(G) \cup\{0\}$, 
which would correspond  to considering the trivial sequence as well. 

As mentioned in the introduction the large cross number of $G$ is defined by
\[
\mathsf K(G)= \max \{\mathsf k(S) \mid S \in \mathcal{A}(G)\}
\]
and the small cross number of $G$, for $|G| \neq 1$,  by
\[
\mathsf k(G)= \max \{\mathsf k(S) \mid S \in \mathcal{A}^{\ast}(G)\}.
\]
In other words, $\mathsf{K}(G) = \max \mathsf{W}(G) $ and $\mathsf{k}(G) = \max \mathsf{w}(G)$.
If $|G|= 1$, then  $\mathsf{K}(G)= 1$ and we set $\mathsf{k}(G)=0$.

Let $G=C_{q_1} \oplus \dots \oplus C_{q_s}$ be a direct sum decomposition of $G$ into cyclic groups of prime power order. Set
\[
\mathsf k^{\ast}(G) = \sum_{i=1}^s \frac{q_i-1}{q_i} \quad and \quad 	\mathsf K^{\ast}(G)=\frac{1}{\exp(G)}+\mathsf k^{\ast}(G) \, .	
\]

Suppose $\{e_1, \dots, e_s\}$ is an independent generating set of $G$ with $\ord(e_i)=q_i$ for each $i \in [1, s]$, then 
$T=e_1^{q_1-1} \dots e_s^{q_s-1}$ is zero-sum free and $S= T(e_1+\dots +e_s) $ is a minimal zero-sum sequence.  
This shows that $\mathsf{K}^{\ast}(G) \le \mathsf{K}(G)$ and $\mathsf{k}^{\ast}(G) \le \mathsf{k}(G)$. In principle one could use any set of independent elements, but using a set with elements of prime power order yields the best bound. 

Given such a lower bound, the question arises if equality holds. Equality holds in particular for $p$-groups (see \cite[Proposition 5.1.8 and Theorem 5.5.9]{Ge-HK06}) and for some other special cases (see e.g. \cite{Ge-Sc94,He14a,Ki15a}); we invoke some results in later sections. No example is known where equality does not hold. Krause and Zahlten conjectured in \cite[Page 688]{Kr-Za91} that the equality $\mathsf K(G)=\mathsf K^{\ast}(G)$ holds for all cyclic groups $G$, but even this remains open.  
 
The two constants $\mathsf{k}(G)$ and $\mathsf{K}(G)$ are closely related. It is easy to see that $\frac{1}{\exp(G)}+ \mathsf{k}(G) \le \mathsf{K}(G)$, yet it is not known if equality always holds. However,  in case $\mathsf{K}(G) = \mathsf{K}^{\ast}(G)$, we have  
\[
\frac{1}{\exp(G)}+\mathsf k^{\ast}(G)\le \frac{1}{\exp(G)}+\mathsf k(G) \leq \mathsf K(G)=\mathsf K^{\ast}(G)= \frac{1}{\exp(G)}+\mathsf k^{\ast}(G)
\]
and thus $\mathsf{k}(G) = \mathsf{k}^{\ast}(G)$ and also $\frac{1}{\exp(G)}+ \mathsf{k}(G) =\mathsf{K}(G)$.

Since we use them sometimes in the current paper, we recall a few results for the Davenport constant. The equality $\mathsf D(G)=|G|$ for finite cyclic groups is known and easy to see.  Moreover for $p$-groups  $G=C_{q_1} \oplus \dots \oplus C_{q_s}$ one has  $\mathsf{D}(G) = 1+ \sum_{i=1}^s (q_i-1)$. See for instance \cite[Chapter  5]{Ge-HK06} for these and further results.

We end this section with some results on $\mathsf{W}(G)$ and $\mathsf{w}(G)$ that we use frequently. 

For every finite abelian group, with $|G| \neq 1$, one has $\mathsf{W}(G) \subseteq \frac{1}{\exp(G)} [2, \exp(G) \mathsf{K}(G)]$ and $\mathsf{w}(G) \subseteq \frac{1}{\exp(G)} [1, \exp(G) \mathsf{k}(G)]$. If $|G|$ is even and $\exp(G)= 2^k m $ with an odd $m$ and $G$ does not contain a subgroup of the form $C_{2^k}^2$, then $\mathsf{W}(G) \subseteq \frac{2}{\exp(G)} [1, \frac{\exp(G)\mathsf{K}(G)}{2}]$, see \cite[Lemma 1]{Ch-Ge96}.

In the other direction for every finite abelian group with $|G| $ odd and  $|G|\neq 1$, one has $\frac{1}{\exp(G)}[2, \exp (G)] \subseteq \mathsf{W}(G)$. 
If $|G|$ is even, then in general one only has  $\frac{2}{\exp(G)}[1, \frac{\exp (G)}{2}] \subseteq \mathsf{W}(G)$; however, if $\exp(G)= 2^k m $ with an odd $m$ and $G$  contains a subgroup of the form $C_{2^k}^2$, then we have  $\frac{1}{\exp(G)}[2, \exp (G)] \subseteq \mathsf{W}(G)$ as in the case of groups of odd order, see \cite[Theorem 2]{Ch-Ge96}.

\section{Results on $\mathsf{w}(G)$} 

In the current section, we obtain some results on the structure of the set $\mathsf{w}(G)$. More specifically, we show in two different ways that the set is in some sense close to an arithmetic progression, and that deviations can only occur for values close to the maximum $\mathsf{k}(G)$. In later sections, we discuss that while for $G$ a $p$-group and in a few other cases, the set is indeed an arithmetic progression, this is not always the case. 

First, we establish and recall a few simple lemmas. It is easy to see that $\mathsf{w}(G)$ contains all small elements; the result for $\mathsf{W}(G)$ is recalled at the end of the preliminaries.  

\begin{lemma} \label{w(k<=1)}
Let $G$ be a finite abelian group. Then
\[ \frac{1}{\exp(G)}[1, \exp(G)-1] \subseteq \mathsf w(G). \]
\end{lemma}

\begin{proof}
We need to show that for each $j \in [1, \exp(G)-1]$, there is some $S \in \mathcal{A}^{\ast}(G)$ with $\mathsf k(S)=\frac{j}{\exp(G)}$.  
Assume $g \in G$ with $\ord(g)=\exp(G)$. Then $S_j=g^j$ is a zero-sum free sequence for each $j \in [1, \exp(G)-1]$ and $\mathsf k(S_j)=\frac{j}{\exp(G)}$.
\end{proof}

While this result is basic, it is indeed sharp in some cases, namely when $G$ is cyclic group of prime power order as in this case, and in this case only, we have $\mathsf{k}(G)= \frac{\exp(G)-1}{\exp(G)}$. 
The following lemma shows that this is indeed the only case.  Later, we establish results that yield the existence of  larger arithmetic progressions in $\mathsf{w}(G)$ in case the group has a large rank.

\begin{lemma} \label{3.1}
Let $G$ be a finite abelian group and let $H \subseteq G$ be a subgroup.
Then $\mathsf w (H) \subseteq \mathsf w (G)$, and equality holds if and only if $H = G$.
\end{lemma}

\begin{proof}
It is immediate from the definion that  $\mathcal{A}^{\ast}(H) \subseteq \mathcal{A}^{\ast}(G)$ and thus $\mathsf w(H) \subseteq \mathsf w(G)$ by definition. Cearly, $H=G$ yields $\mathsf w (H) = \mathsf w (G)$ and it remains to show the converse. Note by \cite[Proposition 5.1.11]{Ge-HK06} that
	\[
	\mathsf k(G)\geq \mathsf k (H) + \frac{\mathsf k (G/H)}{\exp (H)} \,.
	\]
Since $\mathsf w(H)=\mathsf w(G)$ implies $\mathsf k(H)=\mathsf k(G)$, it means $\mathsf k(G/H)=0$, and hence $|G/H|=1$ (which is same as $H=G$).
\end{proof}

The following lemma is a key tool for the current section.
\begin{lemma} \label{3.3}
Let $G_1$ and $G_2$ be non-trivial finite abelian groups. Then
	\[
	\mathsf w^{\circ} (G_1) + \mathsf w^{\circ} (G_2)  \subseteq \mathsf w^{\circ} (G_1 \oplus G_2)  
	\text{ and }
	\mathsf w (G_1) + \mathsf w^{\circ} (G_2)  \subseteq \mathsf w (G_1 \oplus G_2) \,.
	\]
\end{lemma}
\begin{proof}
For $i \in [1, 2]$, let $q_i \in \mathsf w^{\circ} (G_i) $. Since $\mathsf w^{\circ} (G_i) \subseteq \mathsf w^{\circ} (G_1 \oplus G_2)$ by Lemma \ref{3.1}, we may assume that $q_1 > 0 $ and $q_2 > 0$. Then there are zero-sum free sequences $S_i \in \mathcal F (G_i)$ with $\mathsf k (S_i)=q_i$. Since $S_1S_2 \in \mathcal F (G_1 \oplus G_2)$ is zero-sum free, we obtain that
\[
	q_1+q_2 = \mathsf k (S_1)+\mathsf k (S_2) = \mathsf k (S_1S_2) \in \mathsf w^{\circ} (G_1 \oplus G_2) \,. 
\]
The second claim is obvious from the first, since $0$ is not contained in the left-hand set.
\end{proof}

We now combine these results with a result from additive combinatorics to show that in certain cases  $ \mathsf w (G)$ is close to an arithmetic progression.

\begin{theorem} \label{3.4}
Let $G$ be a finite abelian group with $\exp (G)=n \ge 2$. 
\begin{enumerate}
\item If $G = G_1 \oplus \ldots \oplus G_s$, where $s \in \N$, $\exp (G_i)=n$ for all $i \in [1,s]$, then
		\[
		\frac{1}{n} [1, (n-1)s] \subseteq \mathsf w (G) \,.
		\]
\item There exist constants $c, s^{\ast} \in \N$ such that, for all $s \ge s^{\ast}$,
\[
\frac{1}{\exp(G)}[1, s \exp (G) \mathsf k (G) - c] \subseteq   \mathsf w (G^s) \,.
\]		
In particular, if $\mathsf k (G^s) = s \mathsf k (G)$, then $\mathsf w (G^s)$ is an arithmetic progression, apart from a globally bounded upper part.
\end{enumerate}
\end{theorem}

 Note that, for every $s \in \N$, we have
\[
\mathsf k^{\ast} (G^s) = s \mathsf k^{\ast} (G) \,.
\]
Since there are no examples with $\mathsf k (G) \ne \mathsf k ^{\ast} (G)$, there are no examples of groups $G$ and integers $s \in \N$,  for which
\[
\mathsf k (G^s) = s \mathsf k (G)
\]
does not hold.

\begin{proof}
1. Since $\frac{1}{n} [1, n-1] \subseteq \mathsf w (G_i)$ for all $i \in [1,s]$, we obtain that
	\[
	\begin{aligned}
		\frac{1}{n} [1, (n-1)s] & = s  \big( \{0\} \cup \frac{1}{n} [1, n-1] \big)  \subseteq \mathsf w^{\circ} (G_1) + \ldots + \mathsf w^{\circ} (G_s) \\
		& \subseteq \mathsf w^{\circ} (G) \,,
	\end{aligned}
	\]

	2. We set $A = \exp (G) \mathsf w^{\circ} (G) \subseteq \N_0$. By \cite[Theorem 1.1]{Na96b}, there exists integer $c \in \N$ such that the $s$-fold sumset of $A$ has the form
	\[
	 A' \uplus [c, s \max A - c  ] \uplus A'' \quad \subseteq \exp (G) \mathsf w^{\circ} (G^k) \,,
	\]
	with $A' \subseteq [0, c-2]$ and $A'' \subseteq s \max A -c + [1, c]$, for all $s \ge  \max \{ 1, (|A|-2)(\max A - 1)\max A \}$; we set $s^{\ast}$ equal to this value.
	
Now suppose in addition that $(n-1)s^{\ast} \ge c$. Then the first assertion implies that $[1, c] \subseteq \exp(G) \mathsf w^{\circ} (G^s)$. Thus, we obtain that
\[
[0, s \max A - c] \uplus A'' \subseteq \exp (G) \mathsf w^{\circ} (G^s) \,,
\]
and hence the assertion follows.
\end{proof}

There are better estimates for the constant $c$, then the one given in \cite{Na96b}; see for example \cite{Gr-Wa21a, Gr-Sh-Wa23a}, however we do not pursue this route for improvement. Instead, we present our second approach to the problem of determining that $\mathsf{w}(G)$ contains all small elements. 
The approach is similar, in that we use again Lemma \ref{3.3}. However, we do not impose anymore that all the groups have the same exponent. 
This has the advantage of being able to apply the result to any group. The drawback is that the application of Lemma \ref{3.3} is less direct. 
To overcome this issue we need a result on set-addition, which we give in Lemma \ref{addition}. 

The main result we obtain is the following theorem. As a corollary we obtain a complete description of $\mathsf{w}(G)$ for $G$ a $p$-group. 
 
\begin{theorem} \label{4.8}
Let $G=H \oplus C_n$ with  $\exp (G)=n$. Then 
	\[ \frac{1}{n} [1, n-1+ n\mathsf{k}^{\ast}(H)] \subseteq \mathsf w(G) \,.\]
\end{theorem}
 
\begin{corollary}\label{w-p groups}
	Let $G=H \oplus C_n$ with  $\exp (G)=n$ a $p$-group. Then
	\[ \mathsf{w} (G) = \frac{1}{n} [1,  n\mathsf{k}(G)] \,.\]
\end{corollary}

The proof of the following lemma is basic and chances are the result is somewhere in the literature, we include it for lack of a suitable reference. 
We recall that $\Delta(A)$ denotes the set of successive distance of $A$, that is for $A  = \{a_1, \dots, a_k\}$ with $a_i < a_{i+1}$ the set is given by $a_{i+1}- a_i$ for $i \in [1,k-1]$. 

\begin{lemma}\label{addition}
Let $A\subseteq \mathbb{Z}$ be a finite set with $|A|\ge 2$ and  $ \max \Delta (A)\le l$.
Then $[0, l-1] + A  = [\min A, \max A + l-1]$.
\end{lemma}
\begin{proof}
Let $n \in  [\min A, \max A + l-1]$ we show that $n \in [0,l-1] + A$. Let $a \in A$ be maximal with $a \le n$, which exists since $\min A \le n$.
We note that $n-a \le l-1$. This is due to the fact that if $a \neq \max A$, then $[a+1,a+l]$ contains an element of $A$ by the condition on $\Delta(A)$, while if $a= \max A$, then it follows from the fact that $n \le \max A + l - 1$.
\end{proof}

In the following result we allow rather arbitrary direct sum decomposition of $G$ into cyclic groups, note though that  the exponent of $G$ is indeed $n$.  
Usually, starting from a given group, the best way to decompose is to impose that each  $q_i$ is a prime power. 

\begin{proposition} \label{w-C_n+C_q's}
	Let $G= \bigoplus_{i=1}^{t}C_{q_i}\oplus C_n $ where each $q_i$ divides $n$. Then
	
	\[  \sum_{i=1}^{t} \frac{1}{q_i}[0, q_i-1] +\mathsf{w}^{\circ}(C_n) \subseteq \mathsf{w}^{\circ} (G) \]
	and
	\[   \sum_{i=1}^{t} \frac{1}{q_i}[0, q_i-1] +\mathsf{w}(C_n) \subseteq \mathsf{w} (G) \]
	
\end{proposition}
\begin{proof}
	By repeated application of Lemma \ref{3.3} we have
	
	\[  \sum_{i=1}^t \mathsf{w}^{\circ}(C_{q_i}) +\mathsf{w}^{\circ}(C_n) \subseteq \mathsf{w}^{\circ} (G) \,. \] 	
	Now, by Lemma \ref{w(k<=1)} we have  $\mathsf{w}^{\circ}(C_{q_i}) \supseteq  \frac{1}{q_i}[0, q_i-1]$ and the claims follow.
\end{proof}

Combining the result with earlier lemmas we get a main technical result of this section. 

\begin{proposition}\label{w-C_n+C_q's(1)}
	Let $G= \bigoplus_{i=1}^{t}C_{q_i} \oplus C_n$ where each $q_i$  divides $n$. Then
	\[ \frac{1}{n} [0, n-1+ n\sum_{i=1}^{t} \frac{q_i-1}{q_i}] \subseteq \mathsf{w}^{\circ} (G) \]
	and
	\[ \frac{1}{n} [1, n-1+ n\sum_{i=1}^{t} \frac{q_i-1}{q_i}] \subseteq \mathsf{w} (G)\]
\end{proposition}
\begin{proof}
	We apply the Proposition \ref{w-C_n+C_q's} to obtain $ \sum_{i=1}^{t} \frac{1}{q_i}[0, q_i-1] +\mathsf{w}^{\circ}(C_n) \subseteq \mathsf{w}^{\circ} (G)$.
	Now by Lemma \ref{w(k<=1)}, we have $\frac{1}{n}[0,n-1] \subseteq \mathsf{w}^{\circ}(C_n)$.
	By Lemma \ref{addition}, we have $\frac{n}{q_i}[0, q_i-1] +[0,n-1]= [0, n-1 + \frac{n(q_i-1)}{q_i} ]$.
	Then the claim follows by a simple induction.
\end{proof}

We now conclude this section with the proof of the main results. 

\begin{proof}[Proof of Theorem \ref{4.8} and Corollary \ref{w-p groups}]
The theorem is a  direct consequence of Proposition \ref{w-C_n+C_q's(1)}; it suffices to impose that $q_i$ is a prime power for each $q_i$ and to recall the definition of $\mathsf{k}^{\ast}(H)$.  To get the corollary it suffices to recall that $\mathsf{k}(G) = \frac{n-1}{n} + \mathsf{k}^{\ast}(H)$ for $G$ a $p$-group.
\end{proof}

\section{Results on $\mathsf{W}(G)$}

The purpose of this section is to obtain results along the lines of the ones for $\mathsf{w}(G)$ presented in the preceding section. We recall the following result due to Chapman and Geroldinger \cite[Theorem 4]{Ch-Ge96} that gives a complete description of $\mathsf{W}(G)$ for $p$-groups, like we established it in Corollary \ref{w-p groups} for $\mathsf{w}(G)$.  It is interesting to observe the difference for $2$-groups.

\begin{theorem} \label{3.2}
Let $G = C_{n_1} \oplus \dots \oplus C_{n_r}$ be a finite abelian $p$-group with $1=n_0  < n_1 \mid \dots \mid n_r$.
\begin{enumerate}
\item Suppose that $p$ is either odd or that $p=2$ with $n_{r-1}=n_r$. Then
\[
\mathsf W (G) = \frac{1}{\exp(G)} [2, \exp (G)\mathsf K (G)] .
\]
\item Suppose $p=2$ and $n_{r-1} < n_r$. Then
\[
\mathsf W (G) = \frac{2}{\exp(G)}[1, \frac{\exp(G)}{2} \mathsf K(G)].
\]
\end{enumerate}
\end{theorem}

The goal of the remainder of the section is to obtain a result like Proposition \ref{w-C_n+C_q's(1)} for $\mathsf{W}(G)$ instead of $\mathsf{w}(G)$. 
The overall strategy of the proof is the same but the problem is more subtle as we do not have a result like Lemma \ref{3.3}.  Moreover, we impose right away that each $q_i$ is a prime power, which anyway is the most relevant case. 
\begin{proposition} \label{W-C_n+C_q's}  
Let $G= \bigoplus_{i=1}^{t}C_{q_i} \oplus C_n$ where each $q_i$  is a prime power that divides $n$. Then	
	\[  \sum_{i=1}^{t} \frac{1}{q_i}[0, q_i-1] +\mathsf W(C_n) \subseteq \mathsf W(G)	\]
\end{proposition}

The proof is similar to that of Proposition \ref{w-C_n+C_q's}, but a problem is that we cannot simply take the concatenation of zero-sum free sequences. Instead we need to maintain the property that the sequence is a minimal zero-sum sequence while at the same time controlling the cross number. Recall that if $U_i$ is a minimal zero-sum sequence over $G_i$ and $g_i \mid U_i$ for $i \in [1,2]$,  
then $(g_1^{-1}U_1)(g_2^{-1}U_2)(g_1 + g_2)$ is a minimal zero-sum sequence over $G_1 \oplus G_2$ of length $|U_1|+|U_2|-1$. However, to control the cross number we need some information on the orders of the elements involved. 

To this end we establish the following lemma. Let $p$ be a prime, we denote by $\mathsf{v}_p(n)$ the $p$-adic valuation, or $p$-valuation for short, of a positive natural number $n$.   

\begin{lemma}\label{val lemma}
Let $n \ge 2$. For each $w \in \mathsf W(C_n)$, there is an $A \in \mathcal{A}(C_n)$ with $\mathsf k(A)=w$ and for each prime divisor $p$ of  $n$ there exists a $g \in \supp (A)$ such that $\mathsf v_p(\ord(g))=\mathsf v_p(n)$.
\end{lemma}
\begin{proof}
Let $w \in \mathsf W(C_n)$. Then there exists a $T \in \mathcal{A}(C_n)$ with $\mathsf k(T)=w$. Let $p_0\mid n$ be a prime and let $g \in \supp(T)$ be an element whose order has maximal $p_0$-valuation among the orders of elements in $\supp(T)$. If  $\mathsf v_{p_0}(\ord(g))=\mathsf v_{p_0}(n)$, there is nothing to do. Thus, assume that  $\mathsf v_{p_0}(\ord(g))<\mathsf v_{p_0}(n)$. There exists some $g_0 \in C_n$ such that $p_0g_0=g$. We have $\ord(g_0)=p_0\ord(g)$ and this means $\mathsf v_{p_0}(\ord(g_0))=\mathsf v_{p_0}(\ord(g))+1$, while $\mathsf v_{p}(\ord(g_0))=\mathsf v_{p}(\ord(g))$ for every other prime $p$.  Now, the sequence $S=(Tg^{-1})(g_0)^{p_0}$ is a minimal zero-sum sequence with the same cross number $\mathsf k(S)=\mathsf k(T)- \frac{1}{\ord(g)}+ \frac{p_0}{\ord(g_0)}=\mathsf k(T)=w$.
	
Now if $\mathsf v_{p_0}(\ord(g_0))=\mathsf v_{p_0}(n)$, we are done otherwise we repeat the same steps starting with  $S$ instead of $T$ until  we obtain a sequence that contains an element whose order has $p_0$-valuation equal to that of $n$. 

Repeating the argument for each prime $p_0$  dividing $n$, the assertion follows. Note that the above mentioned process only affects the valuations for the prime $p_0$, whence there is no interference between the different steps. 
\end{proof}

With this lemma at hand we can prove the proposition. 

\begin{proof}[Proof of Proposition \ref{W-C_n+C_q's}]
Assume $ C_ {q_i}=\langle e_i \rangle$ where $\ord(e_i)=q_i$ for $i \in [1, t]$. 
Let $[1,t]= \uplus_{j=1}^s I_j$ where the order of $q_i$ for $i \in I_j$ is a $p_j$-power for some prime $p_j$; 
in addition, we assume that all the $I_j$ are non-trivial and the $p_j$ are pairwise distinct. 
In other words, $\oplus_{j=1}^s (\oplus_{i \in I_j}  \langle e_i \rangle)$ is a decomposition into $p$-groups of $\bigoplus_{i=1}^{t}C_{q_i}$, 
and $|I_j|$ is the $p_j$-rank of the group. 

We have 
\[
G= C_n \oplus  \bigoplus_{j=1}^s (\oplus_{i \in I_j}  \langle e_i \rangle). 
\]
Let $w_G=w_c+ \sum_{i=1}^{t} \frac{j_i}{q_i} \in \mathsf W(C_n)+ \sum_{i=1}^{t} \frac{1}{q_i}[0, q_i-1] $ where $w_c \in \mathsf{W}(C_n)$ and $j_i \in [0, q_i-1]$, which we can also write as 		
\[
w_G=w_c+ \sum_{j=1}^s \sum_{i\in I_j}  \frac{j_i}{q_i}.
\]
Then by Lemma \ref{val lemma}, there exists some $S_c \in \mathcal{A}(C_n)$ such that $\mathsf k(S_c) =w_c$ and for each $i \in [1,s]$ the sequence $S_c$ contains an element $g_i$ whose   $p_i$-valuation  is maximal, that is,  its $p_i$-valuation is $\mathsf{v}_{p_i}(n)$. Note that the $g_i$ are not necessarily distinct, which causes slight complication. 

We now  construct, recursively, a sequence with the desired cross number. For clarity, we present the first step in detail and then briefly line out the general step. Let $g_{1}$ be an element in $S_c$ such that the $p_1$-valuation of $g_1$ is maximal. Let
\[ 
g_1^{\ast}=g_{1}- \sum_{i\in I_1}  j_{i}e_{i},
\]
then $\ord(g_1^{\ast})=\ord (g_1)$. The sequence
\[
S_1= g_1^{-1}g_1^{\ast}S_c \prod_{i\in I_1}  e_{i}^{j_i}
\]
is a minimal zero-sum sequence over $C_n \oplus (\oplus_{i \in I_1}  \langle e_i \rangle)$ and $\mathsf k(S_1)= \mathsf k(S_c)-  \frac{1}{\ord(g_1)} + \frac{1}{\ord (g_1^{\ast})}+ \sum_{i \in I_1} \frac{j_i}{q_i}$.
Moreover $S_1$ contains an element whose order has $p_i$-valuation equal to $\mathsf{v}_{p_i}(n)$ for each $i \in [1,s]$. 

The result now follows by repeating this process. Assume that for $k \in [1,s-1]$ we have a sequence $S_k$ over
$C_n \oplus  \bigoplus_{j=1}^k (\oplus_{i \in I_j}  \langle e_i \rangle)$ with cross number $w_c+ \sum_{j=1}^k \sum_{i\in I_j}  \frac{j_i}{q_i}$ that contains an element of whose   $p_i$-valuation  is maximal, that is,  its $p_i$-valuation is $\mathsf{v}_{p_i}(n)$ for each $i \in [1,s]$.

Now let $g_{k+1}$ be in $S_k$ such that the $p_{k+1}$-valuation of $g_{k+1}$ is maximal. Let
\[ 
g_{k+1}^{\ast}=g_{k+1}- \sum_{i\in I_{k+1}}  j_{i}e_{i},
\]
then $\ord(g_{k+1}^{\ast})=\ord (g_{k+1})$. The sequence
\[
S_{k+1}= g_{k+1}^{-1} g_{k+1}^{\ast}S_k \prod_{i\in I_{k+1}}  e_{i}^{j_i}
\]
is a minimal zero-sum sequence over $C_n \oplus    \bigoplus_{j=1}^k (\oplus_{i \in I_j}  \langle e_i \rangle)$ and $\mathsf k(S_{k+1})= \mathsf k(S_k)-   \frac{1}{\ord(g_{k+1})} + \frac{1}{\ord (g_{k+1}^{\ast})}+ \sum_{i \in I_{k+1}} \frac{j_i}{q_i}$.
Moreover $S_{k+1}$ contains an element whose order has $p_i$-valuation equal to $\mathsf{v}_{p_i}(n)$ for each $i \in [1,s]$. 

The proof is complete by observing that $S_s$ has the required property. 
\end{proof}

\begin{theorem} 
Let $G = H \oplus C_n$ be a finite abelian group with $\exp(G)=n$.
\begin{enumerate}
\item If $n$ is odd, then
\[
\frac{1}{n} [2, n+ n\mathsf{k}^{\ast}(H)] \subseteq \mathsf W(G) \,.\]
\item If $n$ is even, then  
\[
\frac{2}{n} [1, \frac{1}{2}(n+ n\mathsf{k}^{\ast}(H))] \subseteq \mathsf W(G) \,.\]\end{enumerate}
\end{theorem}
We note that in case $\mathsf{v}_2 (\exp(H)) < \mathsf{v}_2 (n)$, we have $\mathsf W (G) \subseteq \frac{2}{n}[1, \frac{n}{2} \mathsf K(G)]$ as recalled at the end of Section \ref{sec_prel}, and thus the second part of the result is quite tight, too. However, in case $\mathsf{v}_2 (\exp(H)) = \mathsf{v}_2 (n)$, there is a considerable gap. In Proposition \ref{W(G_p + C_2^r)} we close this gap in a special case.

\section{When are $\mathsf w(G)$  and  $\mathsf W(G)$ arithmetic progressions?}

In earlier sections we showed that for a wide variety of groups the initial parts of the sets $\mathsf w(G)$  and  $\mathsf W(G)$ are arithmetic progressions with difference $\frac{1}{\exp(G)}$ or sometimes $\frac{2}{\exp(G)}$. In the current section we study when the full sets actually are arithmetic progressions. We first recall that this is the case for $p$-groups, more precisely  for $G = C_{n_1} \oplus \ldots \oplus C_{n_r}$ a finite abelian $p$-group with $1=n_0  < n_1 \mid \dots \mid n_r$ we have that $\mathsf{W}(G)$ and $\mathsf{w}(G)$ are arithmetic progressions with difference $\frac{1}{\exp(G)}$, unless $p=2$ and $n_{r-1} < n_r$ in which case it is an arithmetic progression with difference $\frac{2}{\exp(G)}$, see Theorem \ref{3.2}.

We now consider groups that are the direct sum of an elementary $2$-group and a $p$-group and establish in some cases that sets of cross numbers are arithmetic progressions, too. While this is a quite special class of group, there are reasons to assume that the phenomenon is quite rare; we discuss this at the end of this paper.  

We start with a result for cyclic groups; the assertion on $\mathsf{W}(G)$ is known by a result of Baginski et al. \cite[Theorem 2.1]{B-C-M-P04}.

\begin{theorem}\label{w-2p^k}
Let $G=C_{2p^k}$ with $k \in \mathbb{N}$ and $p$ be a prime. Then	
	\[
	\mathsf W(G)=\frac{2}{\exp(G)}[1, \frac{\exp(G)}{2} \mathsf K^{\ast}(G)]  \quad and \quad  \mathsf w(G)=\frac{1}{\exp(G)}[1, \exp(G)\mathsf k^{\ast} (G)] \,.
	\]
\end{theorem}
\begin{proof}
	The statement on $\mathsf W(G)$ was proved by Baginski et al. in \cite[Theorem 2.1]{B-C-M-P04}.   
We need to prove the result for $\mathsf{w}(G)$. 
Based on the outcome concerning $\mathsf{W}(G)$, note that we have $\mathsf k(G)=\mathsf k^{\ast}(G)=\frac{3p^k-2}{2p^k}$. Now for some $g \in G$ with $\ord(g)=2p^k$, take $S=g^i \in \mathcal{A}^{\ast}(G)$ for $i \in [1, 2p^k-1]$ whence
	\[
	\frac{1}{2p^k}	[1, 2p^k-1] \subseteq \mathsf w(G) \subseteq \frac{1}{2p^k} [1, 3p^k-2]
	\]
It remains to obtain the missing elements as cross numbers of a zero-sum free sequence.  Write $G=C_2 \oplus C_{p^k}$ and suppose $(e,f)$ is a generating set of $G$ with $\ord(e)=2$ and $\ord(f)=p^k$. For $1 \leq l \leq \frac{p^k-1}{2}$, set
		\[
	B_l=e(e+f)f^{p^k-1-l} \in \mathcal{A}^{\ast}(G) \quad and \quad B=B_l(e+f)^{-1} \in \mathcal{A}^{\ast}(G)
	\]
	then
	
	\[
	\mathsf k(B_l)=\frac{3p^k-1-2l}{2p^k} \quad and \quad \mathsf k(B)=\mathsf k(B_l)-\frac{1}{2p^k} \,.
	\]
	Altogether, $B_l$ and $B$ give us $[2p^k, 3p^k-3] \subseteq 2p^k \mathsf w(G)$. Now if we put $l=0$ in $B_l$ then $B_0 \in \mathcal{A}(G)$ but $B=B_0(e+f)^{-1} \in \mathcal{A}^{\ast}(G)$ and $k(B)=\frac{3p^k-2}{2p^k}$ and the proof is complete.
\end{proof}

Combining this results with results from the preceding section we obtain more general results.
We start by a result for $\mathsf{w}(G)$, which are again easier to obtain.

\begin{proposition} \label{w(G_p + C_2^r)}
Let $G_p$ be a $p$-group for an odd prime $p$ and let $G= C_2^r \oplus G_p$ for some $r \in \mathbb{N}$.	
Then	
		\[  \frac{1}{\exp(G)} [1, \exp(G) \mathsf k^{\ast}(G)] \subseteq \mathsf w(G) \subseteq \frac{1}{\exp(G)} [1, \exp(G) \mathsf k(G)] 
		\]		
		In particular, if $\mathsf k(G)=\mathsf k^{\ast}(G)$,  then equality holds and $\mathsf{w}(G)$ is an arithmetic progression. 
\end{proposition}

\begin{proof}
The inclusion $\mathsf{w}(G) \subseteq \frac{1}{\exp(G)} [1, \exp(G) \mathsf k(G)] $ is obvious. We show $\frac{1}{\exp(G)} [1, \exp(G) \mathsf k^{\ast}(G)] \subseteq \mathsf{w}(G)$. By Theorem \ref{w-2p^k} the result is true for $\mathsf{w} (C_{\exp(G)})$.  The result now follows by  Proposition \ref{w-C_n+C_q's}.
\end{proof}

We now want a similar result for $\mathsf{W}(G)$. We need a technical lemma. 

\begin{lemma} \label{cyc-p^k-0}  
Let $G$ be a cyclic $p$-group with $G=\langle f\rangle$ where $\ord(f)=p^k$ for some odd prime $p$ and some $k \in \mathbb N$.
\begin{enumerate}
		\item For each $j \in [1, p^k-1]$, there exists some $S_j \in \mathcal{A}^{\ast}(G)$ such that $\sigma(S_j)=-f$ and $\mathsf k(S_j)=\frac{j}{p^k}$.
		\item For each $j \in [1, p^k-2]$, there exists some $T_j \in \mathcal{A}^{\ast}(G)$ such that $\sigma(T_j)=-2f$, $-f \notin \Sigma (T_j)$ and $\mathsf k(T_j)=\frac{j}{p^k}$.
\end{enumerate}
\end{lemma}
\begin{proof}
1. Take $S_j=f^{j-1}((p^k-j)f)$ for some $j \in [1, p^k-1]$, then $S_j$ is a zero-sum free sequence with sum $-f$. 
We have $\mathsf k(S_j)=\frac{j}{p^k}$ unless $p^k-j$ is divisible by $p$. 
Suppose $p^k-j$ is a multiple of $p$, that means $j$ is a multiple of $p$, but then $j-1$ is not a multiple of $p$. Therefore, $S_j'= f^{j-2}(2f)((p^k-j-1)f)$ is a zero-sum free sequence that fulfills the requirements.

2. Take $T_j=f^{j-1}((p^k-j-1)f)$ for some $j \in [1, p^k-2]$, then $T_j$ is zero-sum free with sum $-2f$ and $-f \notin \Sigma(T_j)$. 
The cross number is $\frac{j}{p^k}$ unless $p^k-j-1$ is divisible by $p$. If $p^k-j-1$ is a multiple of $p$, then $p^k-j-2$ is not and therefore $T_j'=f^{j-2}(2f)((p^k-j-2)f)$  has the required properties.
\end{proof}

\begin{proposition} \label{W(G_p + C_2^r)}
Let $G_p$ be a $p$-group for an odd prime $p$ and let $G=C_2^r \oplus G_p$ for some $r \in \mathbb{N}$.
\begin{enumerate}
\item If $r=1$, then
	
	\[   \frac{2}{\exp(G)}[1, \frac{\exp(G)}{2}\mathsf K^{\ast}(G) ]\subseteq     \mathsf W(G) \subseteq  \frac{2}{\exp(G)}[1, \frac{\exp(G)}{2}\mathsf K(G)]	\]	
In particular, if $\mathsf K(G)=\mathsf K^{\ast}(G)$ then equality holds and $\mathsf W(G)$ is an arithmetic progression with difference $\frac{2}{\exp(G)}$.	
\item If $r \geq 2$, then
	
	\[ \frac{1}{\exp(G)} [2, \exp(G) \mathsf K^{\ast}(G)] \subseteq \mathsf W(G) \subseteq \frac{1}{\exp(G)} [2, \exp(G) \mathsf K(G)]
	\]
	
	In particular, if $\mathsf K(G)=\mathsf K^{\ast}(G)$ then equality holds and thus $\mathsf W(G)$ is an arithmetic progression with difference $\frac{1}{\exp(G)}$ .	
\end{enumerate}	
\end{proposition}
\begin{proof}
The inclusions for $\mathsf{W}(G)$ are the well-known ones, see the end of Section \ref{sec_prel}.

It remains to show that $\mathsf{W}(G)$ contains the claimed elements. Assume $G_p= C_{p^{k_1}} \oplus  \dots \oplus C_{p^{k_l}}$ for some $k_i \in \mathbb{N}$, $i \in [1, l]$ and $k_1  \leq \dots \leq k_l$. By Proposition \ref{W-C_n+C_q's}, it suffices to show the results for $G=C_2^r \oplus C_{p^{k_l}}$ only. Therefore, from now on assume $G=C_2^r \oplus C_{p^k}$ for some $k \in \mathbb{N}$. 

The first assertion, then is precisely Theorem \ref{w-2p^k}.

We turn to the second assertion. We show the result for $r=2$ only, again the rest follows from Proposition \ref{W-C_n+C_q's}. Write
\[
G= C_2 \oplus C_2 \oplus C_{p^k} =  \langle e_1 \rangle \oplus \langle e_2 \rangle \oplus \langle f \rangle
\]
where $\ord(f)=p^k$ and $\ord(e_i)=2$ for $i\in [1,2]$.

Then $\mathsf K^{\ast}(G)= 2 \frac{1}{2} +\frac{p^k-1}{p^k}+ \frac{1}{2p^k}=2-\frac{1}{2p^k}$. Now,
 $\frac{1}{\exp(G)} [2, \exp(G) ] \subseteq \mathsf W(G)$ by \cite[Theorem 2]{Ch-Ge96}.
For $j \in [1, p^k-1]$, let $A_j=e_1e_2 (e_1+e_2+f)S_j \in \mathcal{A}(G)$ where $S_j \in \mathcal{A}^{\ast}(G)$ such that $\sigma(S_j)=-f$ and $\mathsf k(S_j)=\frac{j}{p^k}$, which exists by Lemma \ref{cyc-p^k-0}.
We have $\mathsf{k}(A_j) = 1 + \frac{1+2j}{2p^k}$ for $j \in [1,p^k-1]$. 

For $j \in [1, p^k-2]$, let $A_j'=e_1e_2 (e_1+f)(e_2+f)T_j \in \mathcal{A}(G)$ where $T_j \in \mathcal{A}^{\ast}(G)$ such that $\sigma(T_j)=-2f$, $-f \notin \Sigma (T_j)$,  and $\mathsf k(T_j)=\frac{j}{p^k}$, which exists by Lemma \ref{cyc-p^k-0}.

We have $\mathsf{k}(A_j') = 1 + \frac{2+2j}{2p^k}$ for $j \in [1,p^k-2]$. 

Whence $1+\frac{1}{2p^k}$ and $1+\frac{2}{2p^k} $ are the only cross numbers which are not yet realized by some atom in $G$ but 
 $S= e_1e_2 (e_1+f) (e_2-f) \in \mathcal{A}(G)$ with $ \mathsf k(S)=1+\frac{1}{2p^k}+\frac{1}{2p^k}$ and 
 $S=(e_1+e_2-f)(e_1+f)(e_2+f)f^{p^k-1} \in \mathcal{A}(G)$ with $\mathsf k(S)=1+\frac{1}{2p^k}$.

Thus we infer that $\frac{1}{2p^k}[2, 2p^k\mathsf K^{\ast}(G) ]\subseteq \mathsf W(G)$. 

The equality, under the assumption of $\mathsf K^{\ast}(G)= \mathsf K(G)$, is now obvious.  

\end{proof}

In order to have unconditional results of the above mentioned type we study $\mathsf{K}(G)$ for these types of groups, which also yields the result for $\mathsf{k}(G)$.  
We stress that for $\mathsf{k}(G)$ more general results are obtained in  \cite[Theorem 7]{Ki15a}, however we could not see how to obtain the result for $\mathsf{K}(G)$ from those results and thus present proofs even if they are quite similar.

\begin{proposition} \label{prop C_2+H}
Let $H$ be a finite abelian group of odd order. If $\mathsf K(H)=\mathsf K^{\ast}(H)$ and $\sum_{d \mid \exp(H)} \frac{1}{d} < 2$, then $\mathsf K(C_2 \oplus H) = \mathsf K^{\ast}(C_2 \oplus H)$.
\end{proposition}
\begin{proof}
Let $C_2=\langle e \rangle$ with $\ord(e)=2$ and let $H$ be a finite abelian group of odd order with $\sum_{d \mid \exp(H)} \frac{1}{d} < 2$ and $\mathsf K^{\ast}(H)=\mathsf K(H)$, which implies in particular that $\mathsf k^{\ast}(H)=\mathsf k(H)$. Let $S \in \mathcal{A}(C_2 \oplus H)$ and assume for a contradiction that 
	\(\mathsf k(S) > \mathsf K^{\ast}(G)\).   
The condition on the cross number gives 	
\[ \mathsf k(S) \geq  \mathsf K^{\ast}(G)+\frac{1}{2\exp(H)}=\mathsf k(H)+ \frac{1}{2}+\frac{1}{\exp(H)}.\]
 
We distinguish two cases. 

\textbf{Case I:} There is an element, say $e$, in $\supp(S)$ with order $2$. 

In this case we have  $\mathsf k(Se^{-1})\geq  \frac{1}{\exp(H)} +\mathsf k(H)=\mathsf K(H)$. Consider $\pi:  C_2 \oplus H \to (C_2 \oplus H)/\langle e \rangle \cong H$. We note that  $\pi(Se^{-1})$ is a zero-sum sequence. 
Additionally, since the sum of $Se^{-1}$ is $e$, it necessarily contains an element of even order.
 
This implies that 
\[\mathsf k(\pi(Se^{-1})) > \mathsf k(Se^{-1})=\mathsf K(H)	,\]
which means $\pi(Se^{-1})$ is not a minimal zero-sum sequence. Yet this yields a contradiction to $S$ being a minmal zero-sum sequence, 
as $Se^{-1}$ has  a proper subsequence with sum $e$ or $0$, which together with $e$ yields a proper zero-sum sequence of $S$.

\textbf{Case II:} There is no element of order $2$ in  $\supp (S)$.  
Let us write \(  S=S_HS_{e+H} \) where $S_H \in \mathcal{F}(H)$ and $S_{eH} \in \mathcal{F}(e+H)$. 
Now for some $g_1g_2 \mid S_{e+H}$ with $\ord(g_1)=\ord(g_2)$, consider $g_1+g_2$. Then $\ord(g_1+g_2) \mid \frac{\ord(g_1)}{2}$, in particular $ \frac{1}{\ord(g_1)}+\frac{1}{\ord(g_2)} \leq \frac{1}{\ord(g_1+g_2)}$. Thus without loss, we can assume that $S_{e+H}$ contains at most $1$ element of each order as if there are two elements of the same order in $S_{e+H}$, we can replace the two by their sum, which maintains the property that the sequence is a minimal zero-sum sequence and does not decrease the cross number.
Thus we assume $S=S_HR$ where $R$ is a  sequence in $e+H$ which contains at most one element of each order. Then
	\[ \mathsf k(R) \leq \sum_{1\neq d \mid \exp(H)} \frac{1}{2d}= \frac{1}{2}\sum_{1\neq d \mid \exp(H)} \frac{1}{d}  \]
	and	
	\[
	\begin{aligned}
		\mathsf k(S_H) & =\mathsf k(S)-\mathsf k(R)\\
		& \geq \mathsf k(H)+ \frac{1}{\exp(H)} +\frac{1}{2} -\mathsf k(R)
	\end{aligned}
	\]
	Now $\sum_{1\neq d \mid \exp(H)} \frac{1}{d} \leq 1$ implies that
	
	\[
	\mathsf k(S_H) \geq \frac{1}{\exp (H)} +\mathsf k(H)
	\]
	which implies  that $S_H$ is not a zero-sum free sequence in $H$. 
The only way how this does not contradict the fact that $S$ is a minimal zero-sum sequence is that $S=S_H$ that is $R$ is trivial.  
Yet, in this case $\mathsf k(S_H) \geq \mathsf k(H)+ \frac{1}{\exp(H)} +\frac{1}{2} > \mathsf k(H)+ \frac{1}{\exp(H)} $, and thus $S$ is not a minimal zero-sum sequence in $C_2 \oplus H$, which again shows that it is not a minimal zero-sum sequence.  
\end{proof}

\smallskip
	For any $p$-group $H=G_p$,  Proposition \ref{prop C_2+H} implies that $\mathsf K^{\ast}(C_2 \oplus G_p) =\mathsf K(C_2 \oplus G_p)$.

\smallskip
\begin{theorem} \label{K(G_p+C_2^2)}
	Let $G=C_2^2 \oplus G_p$ where $G_p$ is a $p$-group for some odd prime $p$. Then $\mathsf K^{\ast}(G)=\mathsf K(G)$.
\end{theorem}
\begin{proof}
	Assume $\exp (G_p)=p^k$ for some $k \in \mathbb{N}$ so $\exp(G)=2p^k$ and $\mathsf K^{\ast}(G)=1+\mathsf k^{\ast}(G_p) +\frac{1}{2p^k} $. We know $\mathsf K^{\ast}(G) \leq \mathsf K(G)$. Assume $S \in \mathcal{A}(G)$ with $\mathsf k(S) > \mathsf K^{\ast}(G)$. We must show that no such $S$ exists. We have $\mathsf k(S) \geq \mathsf K^{\ast}(G)+\frac{1}{2p^k}= \mathsf k(G_p)+\frac{1}{p^k}+1$. We distinguish three cases.

	\textbf{Case I:} $\supp(S)$ contains two elements of order $2$, say $e_1$ and $e_2$. Let $C_2^2= \langle e_1, e_2 \rangle$. Then
	
	\[ \mathsf k(S(e_1e_2)^{-1}) \geq \mathsf k^{\ast}(G_p)+\frac{1}{p^k}=\mathsf K^{\ast}(G_p) =\mathsf K(G_p)
	\]
	Now let $\pi: G \to G/\langle e_1, e_2 \rangle \cong G_p$ be the canonical epimorphism then $\pi (S(e_1e_2)^{-1}) \in \mathcal{A}(G/\langle e_1, e_2 \rangle)$ but
	
	\[   \mathsf k(\pi(S(e_1e_2)^{-1}))  > \mathsf k(S(e_1e_2)^{-1}) \geq \mathsf K(G_p)
	\]
because there is at least one element in $S(e_1e_2)^{-1}$ of even order, therefore $\pi(S(e_1e_2)^{-1}) > \mathsf K(G_p)$ contradicting our assumption.

	\textbf{Case II:} $\supp( S)$ contains exactly one element of order $2$, say $e$. Let $\pi_1: G \to G/\langle e \rangle \cong G_p \oplus C_2$ be the canonical epimorphism then $\pi_1 (Se^{-1}) \in \mathcal{A}(G/\langle e \rangle)$ but
	\[
	\begin{aligned}
		\mathsf k(\pi_1 (Se^{-1})) &> \mathsf k(Se^{-1})\\
		&> \mathsf K^{\ast}(G)-\frac{1}{2}\\
		&=\mathsf K^{\ast}(G/\langle e \rangle)\\
		&= \mathsf K(G/ \langle e \rangle)
	\end{aligned}
	\]
	the last equality is due to Proposition \ref{prop C_2+H}, a contradiction.

	\textbf{Case III:} $\supp(S)$ does not contain an element of order $2$. \\
	Now set $S= \prod_{i=1}^{k} S_{p^i} \prod_{j=1}^{k}S_{2p^j}$ where $S_{p^i}$ consists of elements of order $p^i$ for $i \in [1, k]$ and $S_{2p^j}$ consists of elements of order $2p^j$ for $j \in [1,k]$. Note that without loss we can assume $|S_{2p^j}| \leq 3$ for all $j \in [1,k]$ as we can replace two elements of even order whose sum is not of even order by their sum (compare this to the argument in Case II of the proof of Proposition \ref{prop C_2+H}). Then
	
	\[\begin{aligned}
		\mathsf k(\prod_{i=1}^{k} S_{p^i}) &=\mathsf k(S)- \mathsf k(\prod_{j=1}^{k}S_{2p^j})\\
		& \geq \mathsf k(S)- \frac{3}{2} \sum_{j=1}^k p^{-j} \\
		& > \mathsf k(S)- \frac{3}{2} \frac{1}{p-1}\\
		& > \mathsf k(S)-1 \\
		& \geq \mathsf k(G_p)+ \frac{1}{p^k}\\
		&=\mathsf K(G_p)	
	\end{aligned}\]
A contradiction to $S \in \mathcal{A}(G)$ hence the assertion follows.
	
\end{proof}

\begin{corollary}  
Let $G_p$ be a $p$-group for an odd prime $p$ and let $G=C_2^r \oplus G_p$ for some $r \in \mathbb{N}$.
\begin{enumerate}
\item If $r=1$, then	\[  \mathsf W(G) =  \frac{2}{\exp(G)}[1, \frac{\exp(G)}{2}\mathsf K^{\ast}(G) ] .\]	 
\item If $r = 2$, then \[\mathsf W(G) = \frac{1}{\exp(G)} [2, \exp(G) \mathsf K^{\ast}(G)] .\]
\end{enumerate}	
\end{corollary}
\begin{proof}
This follows directly from Propositions \ref{W(G_p + C_2^r)}, \ref{prop C_2+H} and Theorem \ref{K(G_p+C_2^2)}.
\end{proof}

The second point yields a class of groups, other than $p$-groups of odd order for which $\mathsf{W}(G)$ is an arithmetic progressions with difference $\frac{1}{\exp(G)}$. This problem was raised in  \cite[Remark 5]{Ch-Ge96}.

\begin{corollary}  
Let $G_p$ be a $p$-group for an odd prime $p$ and let $G=C_2^r \oplus G_p$ for some $r \in \mathbb{N}$.
 If $r \le 2$, then \[\mathsf w(G) = \frac{1}{\exp(G)} [1, \exp(G) \mathsf k^{\ast}(G)] .\]	
\end{corollary}
\begin{proof}
This follows directly from Propositions \ref{w(G_p + C_2^r)}, \ref{prop C_2+H} and Theorem \ref{K(G_p+C_2^2)}.
\end{proof} 

We end this section with a reflection of the problem to which extent the sets $\mathsf{W}(G)$ and $\mathsf{w}(G)$ determine the structure of $G$.  
Note that if $|G|, |G'| \in \{1,2\}$, then $\mathsf W(G)=\mathsf W(G')=\{1\}$. 

\begin{proposition} \label{exp=exp}
Let $G$ and $G'$ be finite abelian groups.	
\begin{enumerate}
		\item Let $|G|, |G'| \geq 3$. If $\mathsf W(G)=\mathsf W(G')$ then $\exp(G)=\exp(G')$ and $\mathsf K(G)=\mathsf K(G')$.
		\item Let $G$ be a $p$-group for some odd prime $p$. We have $\mathsf W(G)= \mathsf W(G')$ if and only if $\exp (G)=\exp (G')$ and $\mathsf K(G)=\mathsf K(G')$. 
		\item 	Let $G$ be a $p$-group for some prime $p$. We have $\mathsf w(G)=\mathsf w(G')$ if and only if $\exp(G)=\exp(G')$ and $\mathsf k(G)=\mathsf k(G')$.
	\end{enumerate}
\end{proposition}

\begin{proof}
	1. Let $\mathsf W(G)=\mathsf W(G')$. Since $\min \mathsf W(G)=\frac{2}{\exp(G)}$ and $\min \mathsf W(G')=\frac{2}{\exp(G')}$, we get from $\mathsf W(G)=\mathsf W(G')$ that $\exp(G)=\exp(G)$. Similarly, since $\max \mathsf W(G)=\mathsf K(G)$ and $\max \mathsf W(G)=\mathsf K(G)$, we get $\mathsf K(G)=\mathsf K(G)$.

	2. Assume first that $\exp(G)= p^k$ for some $k \in \mathbb{N}$. We only need to show the reverse implication of 1. and then the assertion is complete. If $\exp (G)=\exp(G')$, then $\exp(G')=p^k$ and $G'$ is also a $p$-group. Now let $\mathsf K(G)= \mathsf K(G')$. Then by Theorem \ref{3.2}, every possible value between $\min \mathsf W(G)$ and $\max \mathsf W(G)$ can be realized by a minimal zero-sum sequence over $G$, therefore $\min \mathsf W(G) =\frac{2}{\exp (G)}= \min \mathsf W(G')$ and $\max W(G)=\mathsf K(G)=\max \mathsf W(G')$ whence $\mathsf W(G)=\mathsf W(G')$.

	3. 	Suppose $\mathsf w(G)=\mathsf w(G')$. Then $\frac{1}{\exp(G)}=\min \mathsf w(G)=\min \mathsf w(G')=\frac{1}{\exp(G')}$ whence $\exp(G)=\exp(G')$. Moreover, we also get $\mathsf k(G)=\max\mathsf w(G) =\max \mathsf w(G)=\mathsf k(G')$. Conversely, suppose $\exp(G)=p^k$ for some $k \geq 1$ and let $p^k=\exp(G)=\exp(G')$ then $G'$ is also a $p$ group having the same exponent as $G$. Now if $\mathsf k(G)=\mathsf k(G')$ too, then Corollary \ref{w-p groups} implies that $\mathsf w(G)=\mathsf w(G')$.
\end{proof}

The following remark highlights some examples that complement the preceding result.  
\begin{remark} \ 
\begin{enumerate}	
		\item Let $p=2$ and $G_p = C_{p^{k_1}} \oplus \ldots \oplus C_{p^{k_r}}$  with $1  \le k_1 \le \ldots \le k_r$. Let $G=G_p$ with $k_{r-1}=k_r$ and $G'=G_p$ with $k_{r-1} \neq k_r$. Then $\exp(G)=\exp(G')$ and for suitable choices of $r$ and the $k_i$ it is possible to have $\mathsf K(G)=\mathsf K(G')$, too. Indeed, for instance take $G=C_4^3$ and $G'=C_2^3\oplus C_4$ then $\exp(G)=\exp(G')=4$ and $\mathsf K(G)=\mathsf K(G')=\frac{10}{4}$. but in any case, Theorem \ref{3.2} tells that $\mathsf W(G) \neq \mathsf W(G')$. Thus the Proposition \ref{exp=exp}.2 only works for odd primes.
		\item For some $p$-groups $G, G'$, the group rank can still differ even if $\exp (G)=\exp (G')$, $\mathsf K(G)=\mathsf K(G')$ and $\mathsf W(G)= \mathsf W(G')$. For instance, if $G=C_3^4 \oplus C_9$ and $G'=C_9^4$ then $\mathsf W(G)= \mathsf W(G')$ by Theorem \ref{3.2} but $G$ and $G'$ have different rank. Therefore, $\exp (G)=\exp (G')$, $\mathsf K(G)=\mathsf K(G')$ and $\mathsf W(G)= \mathsf W(G')$ does not  determine the structure of the group.	
		\end{enumerate}
\end{remark}

\section{Gap structure in $ \mathsf W(C_p^r\oplus C_q^s)$}

The goal of this section is to highlight the fact that there are gaps in the set of cross numbers for the group  $G= C_p^r  \oplus C_q^s$ with $p> q$ odd primes and $r, s \in \mathbb{N}$ at least if $p$ is large relative to $q$. Note that $r=s=1$ is studied in detail in \cite{B-C-M-P04}.

Since it is needed in this section we recall that for a finite abelian group $G$, the constant $\eta(G)$ is defined to be the smallest integer $t$ such that every sequence $S$ over $G$ with $|S| \geq t$ has zero-sum subsequence $T$ with length $|T| \in [1, \exp(G)]$, such a subsequence is called short. It is easy to see that this constant is finite, for example $\exp(G)|G|$ is a trivial upper bound. 
In general, the exact value is unknown. By definition, we have $\mathsf D(G) \leq \eta (G)$. The exact value of $\eta(G)$ is known for groups of rank at most two and it is known that $\eta(G)\le |G|$ always holds, we refer to \cite[Chapter 5]{Ge-HK06} for these and other results. 

The main results are the two corollaries at the end. We start with a few technical results. 
The following lemma is useful to establish a link between $\mathsf{w}(G)$ and $\mathsf{W}(G)$. 
\begin{lemma} \label{large order}
Let $G=C_p^r \oplus C_q^s$ for $r,s \in \mathbb{N}$ and $p >q$ be primes. Let $S \in \mathcal A(G )$ such that
	\[ \mathsf k(S) > \max\{  \frac{1+r(p-1)}{p}, \frac{1+s(q-1)}{q}  \}.\]
Then $S$ contains an element of order $pq$. 	
\end{lemma}
\begin{proof}
Assume to the contrary that $S$ does not contain an element of order $pq$. Then $S$ contains either only elements of oder $p$ or element of order $q$. 
This contradicts the assumption on the cross number.  
\end{proof}

The following result allows to focus on $\mathsf{w}(G)$ when trying to show the existence of gaps.  

\begin{proposition} \label{connection between W and w}
Let $G=C_p^r \oplus C_q^s$ for $r,s \in \mathbb{N}$ and $p >q$ be primes. 
Then $\mathsf{W}(G) \subseteq  \frac{1}{pq}+ \mathsf{w}(G)$. 	
\end{proposition}
\begin{proof}
Let $S \in \mathcal{A}(G)$. We need to show that $k(S) \in  \frac{1}{pq}+ \mathsf{w}(G)$. If $\mathsf{k}(S) > \max\{  \frac{1+r(p-1)}{p}, \frac{1+s(q-1)}{q}  \}$, then by Lemma \ref{large order}  $S$ contains an element $g$ of order $pq$ and thus $g^{-1}S$ is zero-sum free with cross number $\mathsf{k}(S)-\frac{1}{pq}$, which proves the claim.  If $\mathsf{k}(S) \le \max\{  \frac{1+r(p-1)}{p}, \frac{1+s(q-1)}{q}  \}$, then $\mathsf{k}(S) $ is in $\frac{1}{pq}+ \mathsf{w}(G)$, simply as by  Theorem \ref{4.8} the set contains every integral multiple of $\frac{1}{pq}$ from $\frac{2}{pq}$ up to $\frac{pq+(r-1)(p-1)+ (s-1)(q-1)}{pq}$.   
\end{proof}

The following result gives some insight into the structure of zero-sum free sequences of large cross number for groups of exponent $pq$ for two odd primes $p,q$ under the assumption that $p$ is larger than $q$; this assumption is implicit in the condition on $t$, if $p$ is not large enough no such $t$ exists and the result is void. In other words the result is only relevant for $p > \eta(C_q^s)$.

\begin{proposition} \label{C_p^r+C_q^s}
Let $G=C_p^r \oplus C_q^s$ for $r,s \in \mathbb{N}$ and $p >q$ primes. Let $S \in \mathcal A^{\ast}(G)$ such that
\[ \mathsf k(S) \geq \mathsf k^{\ast}(G)- \frac{t}{pq}\]
where $t\in \mathbb{N}$ such that  $t \le p - \eta(C_q^s) $ . Then $S$ contains $s(q-1)$ elements of order $q$.
\end{proposition}

We first establish a lemma for sequences that do not contain an element of order at most $q$.  

\begin{lemma} \label{lem_C_p^r+C_q^s}
Let $G=C_p^r \oplus C_q^s$ for $r,s \in \mathbb{N}$ and $p >q$ be primes. Let $S \in \mathcal F(G\setminus C_q^s)$ such that
	\[ \mathsf k(S) \geq   \frac{r(p-1)}{p} + \frac{\eta(C_q^s)}{pq}.\]
Then $S$ is not zero-sum free. 
\end{lemma}
\begin{proof}
We write $S = S_pS_{pq}$ where the elements in $S_p$ and $S_{pq}$ have order $p$ and $pq$, respectively. 

Let us write $S_{pq} = Q_1 \dots Q_l R$ such that $\sigma(Q_i) \in C_p^r$ and $|Q_i|\le q$ and $l$ is maximal. In other words, the images of sequences $Q_i$ are short zero-sum sequences over $G/C_p^r \cong C_q^s$. Thus, we have that $|R|\le  \eta(C_q^s)-1$.  
Note that	
\[
pq \mathsf k(S_p S_{pq})=q|S_p|+|S_{pq}|  \leq q|S_p| + ql  + \eta(C_q^s)-1 
\]
and $pq \mathsf k(S_p S_{pq}) \ge q r (p-1) + \eta(C_q^s)$.
Therefore $|S_p|+ l > r(p-1)$. Thus, $S_p \sigma(Q_1) \dots \sigma(Q_l)$ is a sequence over $C_p^r$ of length greater than  $r(p-1)$ and hence cannot be zero-sum free. It follows that $S_p Q_1 \dots Q_l$ is not zero-sum free either, which shows the claim.  
\end{proof}

\begin{proof}[Proof of Proposition \ref{C_p^r+C_q^s}]
We write $S = S_p S_q S_{pq}$ where the order of elements in $S_j$ is $j$ for each $j \in \{p,q,pq\}$. 
Since $S$ is zero-sum free it is clear that $|S_q| $ is at most $s(q-1)$. Assume $|S_q|< s(q-1)$. 
Then $k(S_pS_{pq})  \ge \frac{r(p-1)}{p} + \frac{1}{q} - \frac{t}{pq} $, 
since $p -t \ge\eta(C_q^s) $ by assumption, it follows that  $ \mathsf{k}(S_pS_{pq}) \ge \frac{r(p-1)}{p} + \frac{\eta(C_q^s)}{pq}$. 
Now, by Lemma \ref{lem_C_p^r+C_q^s} the sequence   $S_pS_{pq}$  is not zero-sum free and thus $S$ is not zero-sum free. 
This contradiction establishes the claim. 
\end{proof}

We now determine the large values in the set of cross numbers.  

\begin{corollary}\label{gap str remark_w}
Let $G=C_p^r \oplus C_q^s$ for $r,s \in \mathbb{N}$ and $p >q$ odd primes such that $p \ge  \eta(C_q^s)+2q$. 
Then
	\[ \mathsf w(G) \subseteq \frac{1}{pq}[1, pq \mathsf k^{\ast}(G)] \setminus \{\frac{1}{pq}[pq\mathsf k^{\ast}(G)-(2q-3), pq\mathsf k^{\ast}(G)-(q+1)] \cup \frac{1}{pq}[pq\mathsf k^{\ast}(G)-(q-2), pq\mathsf k^{\ast}(G)-1] \}
	\]
while $\frac{1}{pq}\{pq\mathsf k^{\ast}(G)-(2q-2), pq\mathsf k^{\ast}
(G)-q,  pq\mathsf k^{\ast}(G)-(q-1), pq\mathsf k^{\ast}(G)\} \subseteq \mathsf{w}(G)$ and moreover $\mathsf{k}(G)= \mathsf{k}^{\ast}(G)$. 
\end{corollary}
\begin{proof} 
Let $S \in \mathcal{A}^{\ast}(G)$ such that $\mathsf k(S) \geq \mathsf k^{\ast}(G) - \frac{t}{pq}$ with $t \le 2q $. 
Let $S_q, S_p$ and $S_{pq}$ be the subsequence of $S$ of elements of order $q, p$ and $pq$, respectively. By Proposition \ref{C_p^r+C_q^s} we know that $|S_q| = s(q-1)$.
 
It follows that $\Sigma (S_q) = C_q^s \setminus \{0\}$. Assume $\pi: G \to C_p^r$ be the canonical epimorphism.  It follows that  $\pi(S_pS_{pq})$ is zero-sum free in $C_p^r$. We get that $|\pi(S_pS_{pq})| \leq  r(p-1)$. 
This yields the claim on $\mathsf{k}(G)$. 

If $|S_pS_{pq}| \le r(p-1)-2$, then $\mathsf{k}(S) \le \mathsf{k}^{\ast}(G)-\frac{2q}{pq}$. 

Under the assumption that $|S_pS_{pq}| \ge r(p-1)-1$, let us observe possible values for $\mathsf k(S_pS_{pq})$. For ease of notation we set  $M=\frac{r(p-1)}{p} $, so that $\mathsf k^{\ast}(G)=\frac{s(q-1)}{q}+M$. We include all the relevant values in the table. 
Of course, for $|S_{pq}|$ still larger the value of the cross number is too small to be relevant.  
	
	\begin{center}
		\begin{tabular}{ | m{5em} || m{5em}| m{6cm} | }   
			\hline  
			$|S_p|+ |S_{pq}|$ & $|S_{pq}|$ & $\mathsf k(S_pS_{pq})$ \\ 
			\hline
			\multirow{3}{5em} {\\$r(p-1)$ \\ } & $0$ & $\frac{r(p-1)}{p} =M$ \\
			& $1$ & $\frac{r(p-1)}{p}-\frac{1}{p}+ \frac{1}{pq}=M-\frac{q-1}{pq}$ \\ 
			& $2$ & $\frac{r(p-1)}{p}- \frac{2}{p}+ \frac{2}{pq}=M-\frac{2(q-1)}{pq}$\\  
			\hline 
			\multirow{3}{5em} {\\$r(p-1)-1$\\ } & $0$ & $\frac{r(p-1)}{p}- \frac{1}{p}=M-\frac{q}{pq}$ \\
			& $1$ & $\frac{r(p-1)}{p}- \frac{2}{p}+ \frac{1}{pq}=M-\frac{2q-1}{pq}$ \\ 
			& $2$ & $\frac{r(p-1)}{p}- \frac{3}{p}+ \frac{2}{pq}=M-\frac{3q-2}{pq}$ \\  			
			 
			\hline
		\end{tabular}
		 	\end{center}
	
Thus it can be seen that $\{\mathsf k^{\ast}(G)-\frac{q-2}{pq}, \dots , \mathsf k^{\ast}(G)-\frac{1}{pq}\} \nsubseteq \mathsf w(G)$ and $\{\mathsf k^{\ast}(G)-\frac{2q-3}{pq}, \dots , \mathsf k^{\ast}(G)-\frac{q+1}{pq}\} \nsubseteq \mathsf w(G)$, as those values do not appear in the table. 

Conversely, it is easy to see that zero-sum free sequences with the parameters as indicated in the table actually exist.
Of course a zero-sum free sequence $S_q$ over $C_q^s$ of length $s(q-1)$ and  $S_p'$ over $C_p^r$ of length $r(p-1)$ or $r(p-1)-1$, respectively, exists. The sequence $S_qS_p'$ is zero-sum free.  

Moreover for $h \in C_q^s\setminus \{0\}$ and $g \mid S_p'$, the sequences $g^{-1}(g+h)S_p'$ and $g^{-1}(g+h)S_p'S_q$ are still zero-sum free (the projection of the former to $C_p^r$ remains unchanged and is thus zero-sum free)  and contain exactly one element of order $pq$. Likewise for $g_1g_2 \mid S_p'$ the sequences $g_1^{-1}(g_1+h)g_2^{-1}(g_2+h)S_p'$ and $g_1^{-1}(g_1+h)g_2^{-1}(g_2+h)S_p'$ are still zero-sum free and contain exactly two elements of order $pq$.  
\end{proof}

The proof does not really make use of the fact that $q$ is odd. But the claim is void for $q=2$. The situation is somewhat different for the proof below.

\begin{corollary}\label{gap str remark}
Let $G=C_p^r \oplus C_q^s$ for $r,s \in \mathbb{N}$ and $p >q$ odd primes such that $p \ge  \eta(C_q^s)+2q$. 
Then
	\[ \mathsf W(G) \subseteq \frac{1}{pq}[2, pq \mathsf K^{\ast}(G)] \setminus \{\frac{1}{pq}[pq\mathsf K^{\ast}(G)-(2q-3), pq\mathsf K^{\ast}(G)-(q+1)] \cup \frac{1}{pq}[pq\mathsf K^{\ast}(G)-(q-2), pq\mathsf K^{\ast}(G)-1] \}	 
	\]
while $\frac{1}{pq}\{pq\mathsf K^{\ast}(G)-(2q-2), pq\mathsf K^{\ast}(G)-q,  pq\mathsf K^{\ast}(G)-(q-1), pq\mathsf K^{\ast}(G)\} \subseteq \mathsf{W}(G)$ and moreover $\mathsf{K}(G)= \mathsf{K}^{\ast}(G)$. 
\end{corollary}

\begin{proof} 
We know by Lemma \ref{lem_C_p^r+C_q^s} that $\mathsf{W}(G) \subseteq \frac{1}{pq}+ \mathsf{w}(G)$. 
This proves the claim by invoking Corollary \ref{gap str remark_w}, except for the part regarding the existence of the values. 
However, to see this it suffices to note that we can construct the zero-sum free sequences in the proof of Corollary \ref{gap str remark_w} in such a way that the sum has order $pq$. It is clear that the projection of the sum on $C_p^r$ is non-zero. 
If the projection of the sum on $C_q^s$ is $0$, then we could just use a different element $h$ for our construction; indeed every other element of order $q$ would do in that case, and there is more than one, as $q\neq 2$. 
\end{proof}

It would be possible to obtain more precise results on the sets of cross numbers in case $p$ is much larger than $q$. 
Broadly speaking the gap structure in $\mathsf W(C_p^r\oplus C_q^s)$ follows the same pattern as that in the cyclic case discussed in \cite{B-C-M-P04}. 

We conclude by pointing out that the phenomenon we observed in this section suggest that sets of cross numbers are rarely arithmetic progressions. The point is that zero-sum free sequences whose cross number is maximal should contain only elements whose order is not the exponent. Starting from such a sequence, minor changes will usually not result in a sequence whose cross number is merely $\frac{1}{\exp(G)}$ less than the original sequence. 
  
\section*{Acknowledgments}
The authors are grateful to A. Geroldinger for encouragement and guidance during the work on this paper. In particular, they thank him for his contribution of Theorem \ref{3.4}.

\providecommand{\bysame}{\leavevmode\hbox to3em{\hrulefill}\thinspace}
\providecommand{\MR}{\relax\ifhmode\unskip\space\fi MR }
\providecommand{\MRhref}[2]{%
  \href{http://www.ams.org/mathscinet-getitem?mr=#1}{#2}
}
\providecommand{\href}[2]{#2}

\end{document}